\documentclass[12pt,leqno]{article}

\usepackage{url}

\usepackage{times}
 
\usepackage{amsmath}
\usepackage{amsthm}
\usepackage{amsfonts}
\usepackage{amssymb}

\usepackage{enumerate}

\usepackage[utf8]{inputenc}
\usepackage[OT2,T1]{fontenc}
\DeclareSymbolFont{cyrletters}{OT2}{wncyr}{m}{n}
\DeclareMathSymbol{\Sha}{\mathalpha}{cyrletters}{"58}

\usepackage{comment}

\newtheorem{thm}{Theorem}[section]
\newtheorem{lemma}[thm]{Lemma}
\newtheorem{prop}[thm]{Proposition}
\newtheorem{cor}[thm]{Corollary}

\usepackage{hyperref}

\usepackage{amscd}

\usepackage[mathscr]{eucal}
\usepackage[british]{babel}

\usepackage{indentfirst}

\usepackage{footmisc}

\setfnsymbol{wiley}

\newcommand{\Selp}{\mathrm{Sel}_{p^\infty}}

\newcommand{\QQ}{\mathbb{Q}}
\newcommand{\QQc}{\mathbb{Q}(\mu_{p^\infty})}
\newcommand{\Kc}{K^{\mathrm{cyc}}}
\newcommand{\Kcc}{K_\infty}

\newcommand{\QQdiv}{\mathbb{Q}\left(E\left[p^\infty \right]\right)}
\newcommand{\QQd}{\mathbb{Q}\left(E\left[p\right]\right)}

\newcommand{\Fp}{\mathbb{F}_p}

\newcommand{\ZZ}{\mathbb{Z}}
\newcommand{\CC}{\mathbb{C}}
\newcommand{\Zp}{\mathbb{Z}_p}

\newcommand{\Gal}{\operatorname{Gal}}
\newcommand{\rk}{\mathrm{rk}}

\newcommand{\GL}{\mathrm{GL}}

\newtheorem*{remark}{Remark}

\newif\ifapx
\newif\ifkron
\newif\ifexample

\begin{document}

\title{Ranks of $GL_2$ Iwasawa modules of elliptic curves }
\author{Tibor Backhausz \footnote{student, E\"otv\"os Lor\'and University. E\mbox{-}mail: {\tt taback@math.elte.hu}} }

\date{\today}
\maketitle

\begin{abstract}
Let $p \ge 5$ be a prime and $E$ an elliptic curve without complex multiplication and let $\Kcc=\QQdiv$ be a pro-$p$ Galois extension
over a number field $K$.
We consider $X(E/\Kcc)$, the Pontryagin dual of the $p$-Selmer group $\Selp(E/\Kcc)$.
The size of this module is roughly measured by its rank $\tau$ over a $p$-adic Galois group algebra $\Lambda(H)$, which has been studied in the past decade.
We prove $\tau \ge 2$ for almost every elliptic curve under standard assumptions.
We find that $\tau = 1$ and $j \notin \ZZ$ is impossible, while $\tau = 1$ and $j \in \ZZ$ can occur in at most $8$ explicitly known elliptic curves. 
The rarity of $\tau=1$ was expected from Iwasawa theory, but the proof is essentially elementary.

It follows from a result of Coates et al.\ that $\tau$ is odd if and only if $[\QQd \colon \QQ]/2$ is odd. 
We show that this is equivalent to $p=7$, $E$ having a $7$-isogeny, a simple condition on the discriminant and local conditions at $2$ and $3$.
Up to isogeny, these curves are parametrised by two rational variables using recent work of Greenberg, Rubin, Silverberg and Stoll.
\end{abstract}

\section{Introduction}

Let $E$ be an elliptic curve defined over $\mathbb{Q}$ with good ordinary reduction at the prime $p\geq 5$ and without complex multiplication. We denote by $X(E/K_\infty)$ the dual Selmer group of $E$ over its associated $p$-division extension $K_\infty:=\mathbb{Q}(E[p^\infty])$. The aim of this paper is to investigate the $\Lambda(H)$-rank of $X(E/K_\infty)$ under certain usual technical conditions that are conjectured to be always satisfied. Here $\Lambda(H)$ denotes the Iwasawa algebra of $H=\Gal(K_\infty/\Kc)$ where $K/\mathbb{Q}$ is a finite extension so that $\Gal(K_\infty/K)$ is pro-$p$, and $K^{cyc}$ is the cyclotomic $\mathbb{Z}_p$-extension of $K$. Our main result is that this $\Lambda(H)$-rank $\tau$ can never be $1$ except possibly for finitely many, explicitly known curves. It was previously proven using Iwasawa theoretic techniques that $\tau\neq 0$, and that $\tau=\lambda+s_{E/\Kc}$. Here $s_{E/\Kc}$ denotes the number of primes in $\Kc$ at which the curve $E$ has split multiplicative reduction and $\lambda$ is the usual $\lambda$-invariant of $E$ over $\Kc$, ie.\ the $\mathbb{Z}_p$-rank of the dual Selmer group $X(E/\Kc)$. We do not use further Iwasawa theory. Instead, the main ingredients are $(i)$ refinements of Serre's \cite{serre} study of division points on $E$; $(ii)$ Mazur's \cite{mazur} result on possible isogenies over $\mathbb{Q}$; and $(iii)$ elementary calculations on the moduli curve $X_0(7)$.  In fact, $\tau$ is very rarely odd as one could expect from the formula $\tau \equiv [K \colon \mathbb{Q}]/2 \equiv [\QQd \colon \QQ] \pmod{2}$ (this follows from \cite{cfks}, we give a simplified proof).

This decides the parity of $[\QQd \colon \QQ]/2$ for a given curve in a computationally easy way (Theorem \ref{mikorparos}), and combining this result with parametrisation from \cite{grss} gives all curves with odd $[\QQd \colon \QQ]/2$  (Theorem \ref{grss_param}). This determines all the curves with odd $\tau$.

Moreover, by the formula $\tau=\lambda+s_{E/K^{cyc}}$ there are two possibilities for $\tau=1$: either $\lambda=0$ and $s_{E/K^{cyc}}=1$, or $\lambda=1$ and $s_{E/K^{cyc}}=0$. We prove that the former never occurs --- all the possible exceptions are in the latter case.

Our results are in some sense negative, as Selmer groups with low $\Lambda(H)$-corank would be easier to test conjectures on. Moreover, using the results in \cite{Z} it can be shown that whenever the $j$-invariant of $E$ is non-integral (or, equivalently, if $s_{E/K^{cyc}}\neq 0$) then $X(E/K_\infty)$ is not annihilated by any central element in $\Lambda(G)$ where $G=\Gal(K_\infty/K)$. Combining this with results in \cite{ardakov} would give the first example of a completely faithful Selmer group over the $GL_2$-extension if $\tau=1$. However, as we show, this does not exist in nature even though it is expected that Selmer groups are \emph{all} completely faithful. The possible exceptions are still good candidates to test this and other conjectures. On the other hand, the $\Lambda(H)$-rank encodes important information on the growth of the $\lambda$-invariant inside $K_\infty$ and therefore interesting on its own (see Proposition \ref{lambdanov}).

\textbf{Acknowledgement.} The author would like to thank Gergely Z\'abr\'adi for his indispensable guidance that led to this work.

\section{Assumptions and definitions}
In this section we describe some of our assumptions for a field $K$, prime $p$ and elliptic curve $E$.
\emph{We assume that $E$ does not have complex multiplication.} For CM curves, the theory is different and better understood.

For $G$, a pro-$p$ group with $p$-adic Lie-group structure and no element of order $p$, we define its Iwasawa algebra as the 
inverse limit of $p$-adic group rings
$$\Lambda(G)=\varprojlim \Zp[G/H]$$
where $H$ varies over open normal subgroups of $G$.

For a $\Lambda(G)$-module $M$, the standard definition of rank is
$$\rk_{\Lambda(G)}(M)=\operatorname{dim}_{K(G)} K(G) \otimes_{\Lambda(G)} M$$
where $K(G)$ is the skew field of fractions of $\Lambda(G)$.
Let $\Kcc=\QQ(E[p^\infty])$, and $K$ be a Galois number field such that $\Kcc/K$ is pro-$p$.
Recall that by the Weil pairing, we have $\bigwedge^2 E[p^n] = \mu_{p^n}$,
the group of $p^n$-th roots of unity as a Galois module.
Therefore $K(E[p^n])$ contains $K(\mu_{p^n})$ so $\Kcc$ contains
$\Kc=K(\mu_{p^\infty})$.

Define $G=\Gal(\Kcc/K)$, $H=\Gal(\Kcc/\Kc)$ and $\Gamma=\Gal(\Kc/K)$.

Let $M(p)$ denote the $p$-primary torsion subgroup of a module $M$.
Let $\mathfrak{M}_H(G)$ denote the category of finitely generated $\Lambda(G)$ modules $M$ such that $M/M(p)$ is finitely generated over $\Lambda(H)$.
We make the following assumptions, which are traditional in non-commutative Iwasawa theory.

\begin{enumerate}[(I)]
\item $p \ge 5$
\item $E/K$ has good ordinary reduction at all places above $p$
\item $\Gal(\Kcc/K)$ is pro-$p$
\item $X(E/\Kcc) \in \mathfrak{M}_H(G)$
\end{enumerate}

It is always conjectured that Assumptions (I)-(II) imply Assumption (IV) \cite[Conjecture 5.1]{CFKSV}.
Equivalently, define $Y(E/\Kcc)=X(E/\Kcc)/X(E/\Kcc)(p)$, then $Y(E/\Kcc)$ should be finitely generated over $\Lambda(H)$. This assumption also implies that $X(E/\Kc)$ is torsion over $\Lambda(\Gamma)$ see \cite[Lemma 5.3]{CFKSV}.

In the usual case when $X(E/\Kc)$ is finitely generated over $\Zp$, $X(E/\Kcc)$ is torsion-free and finitely generated over $\Lambda(H)$, in particular $Y(E/\Kcc)=X(E/\Kcc)$.

\section{The $\tau$ rank}
A proposed analogue to $\lambda$ in the non-commutative case is $$\tau=\rk_{\Lambda(H)} Y(E/\Kcc)$$ (see, e.g.~\cite{cfks}, whose notation $\tau$ we follow). We state some earlier results on $\tau$, originally stated with stronger assumptions, and show they are applicable assuming (I)-(IV).

\begin{thm}[Howson]
\label{tauform}
Suppose that Assumptions (I)-(IV) hold. Then
$$\tau = \rk_{\Lambda(H)} Y(E/\Kcc)=\lambda+s_{E/\Kc}$$
\end{thm}
\begin{proof}
As stated above, (IV) is stronger than \cite[Conjecture 2.6]{howson} which implies \cite[Conjecture 2.5]{howson} therefore \cite[Theorem 2.8]{howson} is applicable.
This states that $\lambda+s_{E/\Kc}$ is the homological rank of $X(E/\Kcc)$. This equals $\tau$ using \cite[eqn.~47]{howson}.
\end{proof}

Let $\rk_{p}^{\mathrm Sel} E/F=\rk_{\Zp} X(E/F)$ be the $p$-Selmer rank of $E/F$.

\begin{cor}
\label{taubound}
$\tau \ge {\rk_{p}^{\mathrm Sel} E/K} + s_{E/K}$
\end{cor}
\begin{proof}
This follows from $\lambda \ge \rk_p^{\mathrm Sel} E/K$ \cite[Theorem 1.9]{CIME} and $s_{E/\Kc} \ge s_{E/K}$.
\end{proof}

\begin{thm}[\cite{css}]
Assuming (I)-(IV), $\tau > 0$.
\end{thm}
\begin{proof}
\cite[Theorem 1.5]{CH} means that $Y(E/\Kcc) \neq 0$.
The kernel of the projection $X(E/\Kcc) \to Y(E/\Kcc)$ is finitely generated over $\Lambda(G)$ therefore annullated by some $p^h$.
Let $N$ be a pseudo-null submodule of $Y(E/\Kcc)$ with preimage $M$ in $X(E/\Kcc)$. Then ${p^h}M$ isomorphic to $N$, hence pseudo-null. Under assumptions weaker than Howson's, \cite[Theorem 5.1]{OV} states that all nontrivial pseudo-null submodules of $X(E/\Kcc)$ are zero, hence $N={p^h}M=0$.

Then \cite[Corollary 7.4]{css} holds for $Y(E/\Kcc)$ instead of $X(E/\Kcc)$.
\end{proof}

\begin{prop}[\cite{howson}]
\label{taunov}
If then Assumptions (I)-(IV) for $K$ imply the same for $K'$, and $$\tau(E/K')=[K^{'cyc}:\Kc]\tau(E/K).$$ 
\end{prop}
\begin{proof}
(I) and (II) are obviously unchanged. (III) holds because $G'$ is pro-$p$ as a subgroup of $G$.
Define $G',H'$ analogously for $K'$.
$\Lambda(H)$ is finitely generated of $\Lambda(H')$ rank $[H:H']=[K^{cyc'}:\Kc]$.
\end{proof}

This means that we only need to determine $\tau(E/K)$ when $K$ is minimal among fields satisfying $(I)-(IV)$, and then we can use the above formula. \emph{Therefore from now on we \textbf{assume} that $K$ is minimal in this sense.}

\begin{remark}
Our minimal $K$ will turn out to be same as the field $K$ in \cite{grss} if there is a $p$-isogeny and $\QQd$ otherwise.
\end{remark}

The quantitative meaning of $\tau$ is given by the following,
\begin{prop}[Coates, Howson]
\label{lambdanov}
Assume (I)-(IV).
Let $K_n=\QQ(E[p^n])$. By Serre's theorem \cite{serre} there exists $m$ such that $$\Gal(\Kcc/K_n) \cong \operatorname{ker} \left( GL_2(\Zp) \to GL_2(\ZZ~\mathrm{mod}~p^m) \right).$$
Then $$\lambda(E/K_n)=\tau(E/K_m)p^{3(n-m)}+O(p^{2n}).$$
\end{prop}
\begin{proof}
Take the sequence of subgroups $H_n=\Gal(K_\infty/K_n^{cyc})$. These are $p$-adic Lie groups of dimension $3$, and $|H_m  \colon  H_n|=p^{3(n-m)}$. Then \cite[Corollary 2.12]{howson} means that $\lambda(E/K_n)=\tau(E/K_m)p^{3(n-m)} - s_{E/K_n}$.

The decomposition subgroup $D_q$ of a prime $q$ with multiplicative reduction of $E$ has dimension $2$ as a $p$-adic Lie subgroup of $G$ \cite[Lemma 2.8]{fragments}, therefore these primes each decompose into $O(p^{2n})$ primes over $K_n$. Hence $s_{E/K_n}=O(p^{2n})$.
\end{proof}

Therefore giving a lower bound to $\tau$ implies a lower bound for the growth of $\lambda$ in the tower of division fields of $E$.

\section{The parity of $\tau$}

\begin{thm}
\label{taupar}
Suppose that Conditions (I)-(IV) hold for some Galois number field $K \subseteq \Kcc$.
Then we have 
$$\tau(E/K) \equiv \frac{[K \colon \QQ]}{2} \pmod{2}$$
\end{thm}

\begin{remark}
Theorem \ref{taupar} can be obtained as a consequence of Corollary 5.7. in \cite{cfks} for $F=\QQ$, $F'=K$. Using its notation,
$$\tau \equiv \sum_{\alpha \in \hat{\Omega}, ~ \alpha^2=1} [\mathcal{L} \colon \QQ]/2=|\hat{\Omega}[2]| \cdot [\mathcal{L} \colon \QQ]/2 \equiv |\Omega|  \cdot [\mathcal{L} \colon \QQ]/2 = [K \colon \QQ]/2 \pmod{2}$$
\end{remark}

We give a direct, somewhat simpler proof using Theorem \ref{tauform}, a case of the $p$-parity conjecture and some lemmas about the field $K$. We will use these lemmas in subsequent sections as well.

\begin{prop}
\label{ptorzio}
\begin{enumerate}[(a)]
\item $\Gal(K(E[p])/K)$ has order dividing $p$.
\item $E/K$ has a nontrivial $p$-torsion subgroup.
\end{enumerate}
\end{prop}
\begin{proof}

$\Gal(K(E[p])/K)$ is a factor group of $\Gal(\Kcc/K)$ which is pro-$p$ by our assumptions.
$\Gal(K(E[p])/K)$ acts on $E[p]$ $\Fp$-linearly so it has order dividing $|GL_2(\Fp)|=p(p^2-1)(p-1)$. 
This means $\Gal(K(E[p]/K))$ has $p$-power order dividing $p(p^2-1)(p-1)$, which proves part (a).

If $\Gal(K(E[p])/K))$ is trivial, claim (b) is also trivial.
Otherwise it has order $p$ and so it is a $p$-Sylow subgroup in $\Gal(K(E[p])/K)$. 
As such, it is conjugate to $\left( \begin{smallmatrix} 1 & \Fp \\ 0 & 1 \end{smallmatrix} \right)$ when written in a suitable basis of $E[p]$.
Hence it fixes a one-dimensional $\Fp$-subspace of $E[p]$. 
\end{proof}

\begin{prop}
\label{semstab}
All bad reductions of $E/K$ are split multiplicative.
\end{prop}
\begin{proof}
It is well known that good and split multiplicative reductions remain that way through field extensions so it is enough to
prove the claim for $K$.

When $K$ contains $\QQ(E[p])$, this is a classical result from \cite{ST}.
Otherwise we have by Proposition \ref{ptorzio} part (a) that $\Gal(K(E[p])/K)$ has order $p$.

For places lying above $p$ our assumptions assure good reduction.

Suppose for contradiction that $E/K$ has additive reduction at some $v$ not lying above $p$.
Then \cite[Theorem 1.13.]{lenstraoort} applies and rules out $p$-torsion for $p>3$. This contradicts Proposition \ref{ptorzio}, so $E/K$ is semistable.

Since splitting of a multiplicative reduction depends on solvability of $x^2+c_6$ in the local residue field 
(where $c_6$ is computed from coefficients of $E$). This is unchanged in the degree $p$ extension $K(E[p])/K$, so by \cite{ST}, bad reductions are already split in $K$.
\end{proof}

\begin{prop}
\label{rootnum}
\begin{enumerate}[(a)]
\item The local root number for a place $v$ is
$w_v(E/K)=-1$ if $v$ is Archimedean or $E/K$ has split multiplicative reduction at $v$. Otherwise, $w_v(E/K)=1$.
\item Let $s_{E/K}$ be the number of split multiplicative reductions of $E$ in $K$.
$$w(E/K)=(-1)^{[K \colon \QQ]/2} (-1)^{s_{E/K}}$$

\end{enumerate}
\end{prop}
\begin{proof}
\begin{enumerate}[(a)]
\item For Archimedean and good or split multiplicative non-Archimedean places, this is a special case of Rohrlich's Theorem 2 in \cite{rohr}.
(In his notation, $\tau$ should be the trivial character, and $\chi$ will also be trivial in the split case.)
Other possibilites are ruled out by Proposition \ref{semstab}
\item This follows by multiplying the local root numbers given by (a).
The number of Archimedean valuations of $K$ is $[K  \colon  \QQ]/2$ since $\mu_p \subset K$ so $K$ is totally imaginary.
\end{enumerate}
\end{proof}

\begin{prop}
\label{sparitas}
There are finitely many primes in ${\Kc}$ over any prime of $K$. Furthermore,
$s_{E/\Kc} \equiv s_{E/K} \pmod 2$
\end{prop}
\begin{proof}
$\Gal(\Kc/K) \cong \Zp$ since $\mu_p \subset K$. Then the decomposition subgroup of each prime has finite index, which must be a power of $p$.
Since $p$ is odd, each primes in $K$ corresponds to an odd number of primes in $\Kc$.

\end{proof}

\begin{prop}
\label{ppar}
The $p$-parity conjecture applies for $E/K$ i.e.\  
$(-1)^{\rk_p^{\mathrm Sel} E/K} = w(E/K)$
\end{prop}
\begin{proof}
From Proposition \ref{ptorzio} we have a $p$-torsion subgroup in $E/K$. 
There is a $K$-rational isogeny having this subgroup as kernel.
Then we can apply Theorem 2 from \cite{dokisog}.
\end{proof}

Substituting part (b) from Proposition \ref{rootnum} for the right side of Proposition \ref{ppar}, then using Propositions \ref{sparitas}, we have
$$(-1)^{\rk_p^{\mathrm Sel} E/K}=w(E/K)=(-1)^{[K \colon \QQ]/2} (-1)^{s_{E/K}}$$
$$\rk_p^{\mathrm Sel} E/K + s_{E/K} \equiv {[K \colon \QQ]/2} \pmod 2$$
$$\rk_p^{\mathrm Sel} E/K + s_{E/\Kc} \equiv {[K \colon \QQ]}/{2} \pmod 2.$$
\cite[Proposition 3.10]{CIME} states that $\rk_p^{\mathrm Sel} E/K \equiv \lambda \pmod{2}.$
This proves Theorem \ref{taupar}.

\section{The parity of $[K \colon \QQ]/2$}

Our goal in this section is to classify the elliptic curves $E$ where $[K \colon \QQ]/2$
is odd. This is mostly based on classical results of Mazur and Serre \cite{mazur,serre}. In fact, we roughly follow Serre's argument while also paying attention to parity of various subgroups.
We retain Assumptions (I)-(III). Recall also that $E$ is still assumed to be a non-CM curve defined over $\QQ$.

Note that we assumed in the beginning that $K$ is minimal among fields satisfying Assumption (III). The parity of $\tau$ for other fields in the tower is the same (Proposition \ref{taunov}).

\subsection{Inertia}

In this section, denote $\Gal(\QQd/\QQ)$ by $G$ and $\Gal(K/\QQ)$ by $G_0$.
Since it acts faithfully on $E[p] \cong \Fp \times \Fp$, $G$ is identified with a subgroup of $GL_2(\Fp)$.
For a prime $q \in \QQ$, let $D_q$ denote its decomposition subgroup within $G_0$, 
and let $I_q$ denote its subgroup of inertia within $D_q$. 
(Note that $D_q$ and $I_q$ are, in general, defined only up to conjugacy in $\Gal(K/\QQ)$. 
However, they are unique if the extension is Abelian, which turns out to be the most interesting case.)

Recall that $q$ splits into $[G_0  \colon  D_q]$ distinct prime ideals, and has ramification degree $|I_q|$.
$I_q$ is also a normal subgroup of $D_q$ with a cyclic quotient (isomorphic to the Galois group of an extension of finite fields).

\begin{prop}[{Serre, \cite[Section 1.11]{serre}}]
\label{Ipmatrix}
$I_p$ is either 
\begin{enumerate}[a)]
\item conjugate to a subgroup of the form $\left(\begin{smallmatrix}
1&0\\ 0& \Fp^\times
\end{smallmatrix} \right)$ of order $p-1$. We will call these semi-Cartan subgroups.

\item a non-split Cartan subgroup (isomorphic to a cyclic group of order $p^2-1$, corresponding to 
the action of a primitive root in $\mathbb{F}_{p^2}$ by multiplication)
\end{enumerate}
\end{prop}

Case b) means $4 \mid p^2-1 \mid |G|$ so we can exclude it.

\begin{remark}
Case b) would also contradict Assumption (II) since it implies supersingular reduction at $p$.
\end{remark}

\subsection{Image in $GL_2$ and $PGL_2$}

Serre gives a classification for $\Gal(\QQd/\QQ)$, based on the following definitions. Borel and split Cartan subgroups are defined as conjugate to respectively
$$\left( \begin{matrix} \Fp^\times & \Fp \\ 0 & \Fp^\times \end{matrix} \right) \quad \mathrm{and} \quad \left( \begin{matrix} \Fp^\times & 0 \\ 0 & \Fp^\times \end{matrix} \right).$$
Non-split Cartan subgroups are as defined in Proposition \ref{Ipmatrix}.

\begin{prop}[Serre]
$G$ satisfies at least one of these:
\begin{enumerate}[a)]
\item $G=\GL_2(\Fp)$
\item $G$ is contained in a Borel subgroup
\item $G$ is contained in the normaliser of a split Cartan subgroup
\item $G$ is contained in the normaliser of non-split Cartan subgroup
\end{enumerate}
\end{prop}

Note that it can be easily computed (and Serre does so) that in cases a), c) and d), $p$ does not divide $|G|$. Therefore
\begin{prop}
\label{simabove}
If $p \mid |G|$ but $4 \nmid |G|$ then $G$ is contained in a Borel subgroup.
\end{prop}

\begin{prop}[{Serre \cite[Section 2.6]{serre}}]
\label{projklassz}
Suppose that $p \nmid |G|$ for a group $G < GL_2(\Fp)$.
Let $H$ be the quotient of $G$ by the center of $GL_2(\Fp)$.
Then $H$, lying in $PGL_2(\Fp)$, satisfies at least one of these:
\begin{enumerate}[i)]
\item $H$ is cyclic. Then $G$ is in a Cartan subgroup of $GL_2(\Fp)$.
\item $H$ is dihedral, containing a cyclic subgroup $C$ of index $2$. $C$ is contained in a unique Cartan subgroup of $PGL_2(\Fp)$ normalised by $H$. Then $G$ is in the normaliser of a Cartan subgroup.
\item $H$ is isomorphic to $A_4$, $S_4$ or $A_5$.
\end{enumerate}
\end{prop}

\begin{prop}
\label{csakaborel}
Suppose that $4 \nmid |G|$. Then $G$ lies in a Borel subgroup.
\end{prop}
\begin{proof}
Using Proposition \ref{simabove} we can assume that $p \nmid |G|$.
Then we look at the cases in Proposition \ref{projklassz}.

In case i), the Cartan subgroup containing $G$ is either split or non-split.
If it is split, then it is contained in a Borel subgroup. 
Otherwise $I_p$ must have been a nonsplit Cartan subgroup, which leads to $4 \mid |G|$.

In case ii), by \cite[Proposition 14]{serre} the Cartan subgroup normalised by $G$ contains the semi-Cartan subgroup $I_p$ (see Proposition \ref{Ipmatrix}). We use the basis where $I_p$ is $\left(\begin{smallmatrix}
1&0\\ 0& \Fp^\times
\end{smallmatrix} \right)$. Then the projection of $\left(\begin{smallmatrix}
1&0\\ 0& -1
\end{smallmatrix} \right)$ is in the Cartan subgroup normalised by $H$, thus in the index $2$ cyclic subgroup of $H$. Since it has order $2$, the index $2$ cyclic subgroup of $H$
has even order hence $4 \mid |H| \mid |G|$.

In case iii), it is enough to note that $|A_4|$, $|S_4|$ and $|A_5|$ are all multiples of 4.
\end{proof}

\subsection{Restrictions on $p$}

Whether $G$ is contained in a Borel subgroup is equivalent to whether $E/\QQ$ has an isogeny of degree $p$ to some elliptic curve $E'$.

Mazur's results \cite{mazur} show that a non-CM curve $E/\QQ$ can only have isogenies with prime degree for
$$p \in \{2, 3, 5, 7, 11, 13, 17, 37\}$$

We exclude further primes with the following simple observation.
\begin{prop}
\label{nkp1}
Assume in addition to (I)-(III) that $p \equiv 1 \pmod{4}$. Then
$$[K \colon \QQ]/2 \equiv 0 \pmod{2}$$
\end{prop}
\begin{proof}
From the Weil pairing, $K \ge \QQ(\mu_p)$ so $4 \mid [\QQ(\mu_p) \colon \QQ] \mid [K \colon \QQ]$.
\end{proof}
 
With this and Assumption (I), we can exclude all primes but $7$ and $11$. 

\subsection{Inertia in the Borel case}

Whether $G$ is contained in a Borel subgroup is equivalent to whether $E/\QQ$ has an isogeny of degree $p$.
Borel subgroups can be written over a suitable basis as
$$\left( \begin{matrix} \Fp^\times & \Fp \\ 0 & \Fp^\times \end{matrix} \right)$$

We work in this basis from now on. Note that the Borel subgroup contains the unipotent subgroup (with $1$s in the diagonal)
as a normal subgroup of order $p$.

Recall that we chose $K$ to be the minimal field over which $K_\infty$ is a pro-$p$ extension. Therefore $K$ is contained in the fixed field of the unipotent subgroup of $K$. Then elements of $G_0$ (understood as cosets in $G$)
will be written as $$\left( \begin{matrix} a & \Fp \\ 0 & b \end{matrix} \right)$$
for $a,b \in \Fp^\times$. Note that by some abuse of terminology these have well-defined trace and determinant.

$G_0$ is isomorphic to a subgroup of $\Fp^\times \times \Fp^\times$ and is therefore Abelian. For a rational prime $q$, let $I_q$ be the inertia subgroup of $\Gal(K/\QQ)$ at $q$. 

The isomorphism $$\bigwedge\nolimits^2 E[p] \cong \mu_p$$ implies that the action of $\Gal(K/\QQ)$ on $\mu_p$ is given by the determinant on $G_0$. The kernel of $\det$ is $\Gal(K/\QQ(\mu_p))$.

Since $\QQ(\mu_p) \subset \QQ(K)$, $\mathrm{det}$ is surjective to $\Fp^\times$. Moreover, $\det \colon  I_p \to \Fp^\times$ is a bijection since both have $p-1$ elements (Prop. \ref{Ipmatrix}).

Therefore $\det \colon  G_0 \to \Fp^\times$ belongs to split exact sequence i.e.
$$G_0 \cong \Gal(K/\QQ(\mu_p)) \times I_p.$$

\begin{prop}
\label{paratlaninercia}
$2 \nmid [K \colon \QQ]/2$ is equivalent to $p \equiv 3 \pmod 4$ and $2 \nmid |I_q|$ for all rational primes $q \neq p$.
\end{prop}
\begin{proof}
$[K \colon \QQ]=|G_0|=|I_p| \times \left| \Gal(K/\QQ(\mu_p)) \right|$.

$|I_p|=p-1$ so if $p \equiv 1 \pmod 4$ we are done, and otherwise $[K \colon \QQ]/2 \equiv \left| \Gal(K/\QQ(\mu_p)) \right| \pmod{2}$.

If for any $q \neq p$, $I_q$ is contained in $\Gal(K/\QQ(\mu_p))$ since $\QQ(\mu_p)$ is unramified at $q$. Hence $|I_q| \mid \left| \Gal(K/\QQ(\mu_p)) \right|$.

In the other direction, $I_q$ together generate all of $\Gal(K/\QQ(\mu_p))$  (otherwise $\QQ$ would have an unramified extension), so if each is odd then $\Gal(K/\QQ(\mu_p))$ has odd exponent, therefore also odd order.

\end{proof}

Note that our $I_q$ for a prime $q \neq p$ is the same as Serre's $\phi_q$ (This follows from $p \ge 5$ and {\cite[Proposition 23 (b)]{serre}}).

\begin{prop}[{Serre \cite[Section 5.6 part a)]{serre}}]
\label{minext}
Let $\QQ_q^{\mathrm{unr}}$ be a maximal unramified extension of $\QQ_q$ and suppose
$E$ has potentially good reduction at $q$.
Then $|I_q|$ is the degree of the minimal extension over $\QQ_q^{\mathrm{unr}}$ where $E$ obtains good reduction.
\end{prop}

\begin{prop}
If $E/\QQ$ has additive, potentially multiplicative reduction at $q$, $|I_q|=2$.
\end{prop}
\begin{proof}
$E$ becomes semistable at $q$ at the degree $2$ extension $\QQ_q(\sqrt{-c_6})$ where $c_6$ is a fixed polynomial of the coefficients of $E$
(see \cite{silverman}).
\end{proof}

Therefore, in particular, $|I_q|=1$ is equivalent to $E$ being semistable at $q$.
Serre states that since $E$ obtains good reduction at $q$ (its discrimant $\Delta$ has $q$-valuation $0 {\pmod {12}}$)
 at a field extension with Galois group $I_q$, $|I_q|v_q(\Delta) \equiv 0 \pmod {12}$, 
and if $\mathrm{gcd}(q,12)=1$ then $\mid I_q \mid=\frac{12}{\mathrm{gcd}(12,v_q(\Delta))}$.

Note that by inspecting Serre's list of possibilites in points $a_1)$, $a_2)$ and $a_3)$ of section 5.6. in \cite{serre}, the only odd possibilities for $|I_q|$ are $1$ and $3$.

Summarizing the above, we have the following.
\begin{thm}
\label{mikorparos}
Suppose that $\Gal(\QQd/\QQ)$ is in a Borel subgroup.
If it has order not divisible by $4$, the following conditions hold necessarily.
\begin{enumerate}[1)]
\item $p \equiv 3 \pmod 4$
\item For all primes $q \neq p$ where $E/\QQ$ has additive reduction, it has potentially good reduction and $4 \mid v_q(\Delta)$.
\end{enumerate}
These conditions are sufficient provided $E/\QQ$ is semistable at $2$ and $3$,
 or it is otherwise known that $|I_2|$ and $|I_3|$ are odd.
\end{thm}

Note that these properties can be checked quickly by a computer as long as it can factorise the discriminant.

\begin{prop}
\label{mikortorzio}

Suppose that $E$ is an elliptic curve with a $p$-isogeny and $|I_q|=1$ for all primes $q \neq p$.
Let $E'$ be the $p$-isogeny pair of $E$.
Then either $E$ or $E'$ have rational $p$-torsion.
\end{prop}
\begin{proof}
By \cite[{Proposition 21 ii)}]{serre}, one of
$$0 \to \ZZ/{p\ZZ} \to E[p] \to \mu_p \to 0$$
$$0 \to \mu_p \to E[p] \to \ZZ/{p\ZZ} \to 0$$
is an exact sequence of Galois modules.
These imply that respectively one of
$$0 \to \mu_p \to E'[p] \to \ZZ/{p\ZZ} \to 0$$
$$0 \to \ZZ/{p\ZZ} \to E'[p] \to \mu_p \to 0$$
is an exact sequence as well.
\end{proof}

Now we can rule out $p=11$. The theorem above means that for all $q \neq 11$, $|I_q|$ is $1$ or $3$. The latter is impossible since $I_q$ is a subgroup of $\mathbb{F}_{11}^\times \times \mathbb{F}_{11}^\times$ but $3 \nmid 100$.
Then $I_q=1$ for all $q \neq p$ and the proposition above implies the existence of an elliptic curve with rational $11$-torsion. But there is no such curve by the work of Mazur \cite{mazur}.

We set $p=7$. A result of Greenberg, Rubin, Silverberg and Stoll (Theorem 3.6 in \cite{grss}) gives a parametrisation of all $E$ that have odd
$[\QQd \colon \QQ]/2$, up to isogeny. 
$G_0$ is given by $$\left( \begin{matrix} \chi' & \Fp \\ 0 & \chi'' \end{matrix} \right)$$
for characters $\chi',\chi''$, giving the action on $\operatorname{ker}\varphi$ and $
E[p]/\operatorname{ker}\varphi$ respectively, where $\varphi$ is a $p$-isogeny.

Let $\omega$ denote the character $\Gal(K/\QQ) \to \Fp^\times$ given by action on $\mu_p$.

Then the characters $\chi'$,$\chi''$ restricted to $I_p$ are $\omega^{a'}$ and $\omega^{a''}$ respectively for some $a',a''$.
From the determinant, $\omega^{a'+a''}=\omega$ so one of $a'$ and $a''$ must be even, hence the $p$-inertia part of a character has odd order. Since $|I_q|$ is odd for all
primes $q \neq p$, one of $\chi'$ and $\chi''$ has odd order.
Changing $E$ to its $7$-isogeny pair $E'$ interchanges $\chi'$ and $\chi''$ so up to isogeny, we can assume that the order of $\chi'$ divides $3$. Then we can adapt the theorem almost word by word, setting $k=\QQ$.

\begin{thm}
\label{grss_param}
Let $E/\QQ$ be an elliptic curve and $p$ a prime.
Under Assumptions (I)-(III), $[K \colon \QQ]/2$, equivalently $[\QQd \colon \QQ]/2$, is odd if and only if $E$ has a rational $7$-isogeny and there is a $v \in \QQ$ such that $E$ is {\bf 7-isogenous} over $\QQ$ to the
elliptic curve

$$A_{v,t}  \colon  y^2+ a_1(v,t)xy + a_3(v,t)y = x^3+ a_2(v,t)x^2+ a_4(v,t)x + a_6(v,t)$$

defined as in \cite[Theorem 3.6]{grss}, with an appropriate rational parameter $t$.

\end{thm}

\begin{remark}
Here $t$ determines the character $\chi'$.
\end{remark}

\section{A lower bound for $\tau$}

In this section we establish $\tau \ge 2$ under Assumptions (I)-(IV) and the extra condition $j(E) \notin \ZZ$. Recall that $K$ is the minimal field satisfying Assumption (III). See Proposition \ref{taunov} for other fields in the tower.

Note that $j(E) \notin \ZZ$ guarantees $\tau \ge s_{E/K} \ge 1$ as the denominator of $j(E)$ will be divisible by some prime.

Now suppose $\tau=1$, which is odd, therefore $p=7$ and $E$ has a $7$-isogeny by the previous section.

\subsection{$7$-torsion}
Suppose $|I_q|=1$ for all primes $q \neq 7$, then by Proposition \ref{mikortorzio} $E$ or its isogeny pair $E'$ has rational $7$-torsion.

Let $A \in \{E,E'\}$ be the curve with rational $7$-torsion. Suppose for contradiction that $E$ has good reduction at $2$.
Then its rational $7$-torsion points map injectively to its reduction $\tilde{A}$ over $\mathbb{F}_2$ \cite{silverman}.
Hence $\tilde{A}$ is an elliptic curve with at least $7$ points over $\mathbb{F}_2$. 
But by the Hasse bound, an elliptic curve over a finite field $\mathbb{F}_q$ of order $q$ can have at most $(\sqrt{q}+1)^2$ points and
$(\sqrt{2}+1)^2 \approx 5.82842712 < 7$ which is a contradiction.
A variant of the above argument is given in \cite{serre}.

Therefore $A$ must have semistable bad (i.e.\   multiplicative) reduction at $2$. 
Since the conductor of an elliptic curve is isogeny invariant, $E$ also has multiplicative reduction at $2$.

Over $K=\QQ(\mu_7)$, the prime $2$ decomposes into $2$ primes, and by Proposition \ref{semstab} the reductions at these primes are all
split multiplicative, which gives $2 \le s_{E/K} \le \tau$. Note that from parity, we have in fact $3 \le \tau$. This is attained by the example given in \cite{cfks}.

\subsection{Additive reduction at $q$}
If the above does not hold, there is some prime $q \neq p$ with $|I_q| \neq 1$.

Let $\ell \in \QQ$ be a rational prime dividing the denominator of the $j$-invariant of $E$ 
i.e.\ a prime where $E$ has potentially multiplicative reduction.
By Theorem \ref{mikorparos} this is semistable multiplicative reduction and $|I_\ell|=1$.

We show that $\ell$ must decompose in $K$.
Suppose for contradiction that $\ell$ does not decompose i.e.\   its decomposition subgroup is all of $G_0$. $G_0$ is then the quotient of the decomposition subgroup by $I_\ell$, and as such it should be cyclic. Recall that $|I_q|$ must be a nontrivial factor of $|\Fp^\times|$.
$G_0$ contains $I_q \times I_p$ which cannot be cyclic since $\mathrm{gcd}(|I_q|,|I_p|)=|I_q|>1$.

Therefore there will be at least $3$ primes in $K$ lying over $\ell$. 
These will all have split multiplicative reduction by Proposition \ref{semstab}, hence $3 \le s_{E/K}$ and our claim follows.

\section{Integral $j$-invariant}

Our main tool is the following well known theorem:

\begin{thm}
There is a $p$-isogeny between two elliptic curves $E$ and $E'$ if and only if $(j(E),j(E'))$ is a 
point on the curve $X_0(p)$.
\end{thm}

Using Theorem \ref{grss_param}, we can restrict to $p=7$.
Therefore we are looking for integral points on $X_0(7)$.

\subsection{Integral points on $X_0(7)$}

$X_0(7)$ has genus $0$, therefore it has a rational parametrisation (see, e.g.~\cite{hoeij})
$$\big( (t^2+13t+49)(t^2+245t+2401)^3/t^7, (t^2+13t+49)(t^2+5t+1)^3/t \big) \qquad t \in \QQ$$
We need both coordinates to be integral. Let $t=a/b$ in reduced form.
The first coordinate is then 
$$\frac{(a^2+13ab+49b^2)(a^2+245ab+2401b^2)^3}{a^7 b}$$
Modulo $a$, the numerator is $7^{14} b^8$. Using $(a,b)=1$, this is divisible by $a$ if and only if $a \mid 7^{14}$.
Modulo $b$, the numerator is $a^8$. This is divisible by $b$ if and only if $b \mid 1$.

The second coordinate is
$$\frac{(a^2+13ab+49b^2)(a^2+tab+b^2)^3}{a b^7}$$
Modulo $a$, the numerator is $7^2 b^8$. Using $(a,b)=1$, this is divisible by $a$ if and only if $a \mid 7^2$.
Modulo $b$, the numerator is $a^8$. This is divisible by $b$ if and only if $b \mid 1$.

Therefore the possibilities are $t \in \{1,7,49,-1,-7,-49\}$.
Note that if $t$ parametrizes the pair $(j_1,j_2)$ then $49/t$ gives $(j_2,j_1)$.

$t=1$ and $t=49$ give $j \in \{3^2 \cdot 7 \cdot 2647^3, 3^2 \cdot 7^4 \}$.
$t=-1$ and $t=-49$ give $j \in \{- 3^3 \cdot 37 \cdot 719^3, 3^3 \cdot 37\}$
The rest are symmetric i.e. CM points:
$t=7$ gives $j=3^3 \cdot 5^3 \cdot 17^3$ and $t=-7$ gives $j=- 3^3 \cdot 5^3$.

Therefore the possible $j$-invariants are 
$j \in \{3^2 \cdot 7 \cdot 2647^3, 3^2 \cdot 7^4, -3^3 \cdot 37 \cdot 719^3, 3^3 \cdot 37\}$.

\subsection{Twisting}

Let $E_d$ denote the twist of an elliptic curve $E$ by the character $\left( \frac{d}{\cdot} \right)$ for a square-free integer $d$. 
Explicitly, for an equation
$$E \colon y^2=x^3+a_2 x^2+a_4 x + a_6$$ 
$$E_d \colon d y^2 = x^3+a_2 x^2+a_4 x + a_6$$ 

$E$ and $E_d$ are not isomorphic over $\QQ$ if $d \neq 1$, but they are isomorphic over $\QQ(\sqrt{d})$.

It is well known (see, e.g. \cite{silverman}) that
\begin{thm}
If $E/\QQ$ is an elliptic curve with $j(E) \neq 0,~1728$ then the elliptic curves with $j$-invariant $j(E)$ are 
exactly the curves $E_d$.
\end{thm}

\begin{lemma}
Let $E/\QQ$ be an elliptic curve with a $p$-isogeny, having good reduction at a prime $q \neq p$.
Let $\left( \frac{d}{\cdot}\right)$ be a quadratic character with conductor divisible by $q$.
Then the order of the inertia subgroup of $q$ in $\Gal(\QQ(E_d[p])/\QQ)$ is $2$.

\end{lemma}
\begin{proof}
Since the existence of a $p$-isogeny only depends on the $j$-invariant, $E_d$ also has a $p$-isogeny.
$E/\QQ$ and $E_d/\QQ$ are isomorphic over $\QQ(\sqrt{d})$.
$E/\QQ(\sqrt{d})$ has good reduction at $q$,
therefore so does $E_d/\QQ(\sqrt{d})$.

Hence the minimal extension where $E_d$ obtains good reduction at $q$ is a quadratic extension
with ramification degree $2$ at $q$. The claim follows from Proposition \ref{minext}.

\end{proof}

\begin{prop}
\label{csakveges}
For any given $j$-invariant $j_0$, there are only finitely many curves $E/\QQ$ having $j(E)=j_0$ and also satisfying Assumptions (I)-(IV) and 
$4 \nmid [\QQ(E[p]) \colon \QQ]$. These curves have the same conductor apart from a possible factor of $7^2$.
\end{prop}

\begin{proof}
Let $E/\QQ$ be a curve with minimal conductor $N_E$ among elliptic curves with $j$-invariant $j_0$ and satisfying the conditions.
Then by our previous results, $E$ has a rational $7$-isogeny.

Let $\Delta$ be the minimal discriminant of $E/\QQ$.
If $d$ is a square-free integer not dividing $7\Delta$, then there is prime $q \neq 7$ dividing $d$ where
$E$ has good reduction. Then by the above lemma, $2 \mid |I_q|$ so by Proposition \ref{paratlaninercia}, $4 \mid [\QQ(E_d[7]) \colon \QQ]$.

Therefore all exceptional curves with $j$-invariant $j_0$ are twists of $E$ by some square-free divisor of $7\Delta$, 
of which there are finitely many.

Similarly, twists that change the conductor result in a larger conductor because we chose $N_E$ to be minimal.
The twisted conductor is either $7^2 N_E$ (since additive bad reduction appeared at $p$) or has a prime divisor $q \neq 7$ where $E$ has good reduction.
This implies a good reduction becomes potentially good additive after the twist, and we can invoke the lemma. Note that since $p \ge 5$ the exponent
of $p$ in the conductor of any elliptic curve with integral $j$-invariant is $0$ or $2$.

\end{proof}

\subsection{Calculations}

Together with the list of possible $j$-invariants, Proposition \ref{csakveges} provides a list of all curves that could have $\tau=1$. Using the SAGE \cite{sage} function {\tt EllipticCurve\_from\_j}, we obtain a minimal conductor elliptic curve for each $j$-invariant involved. We take all curves with these conductors and also their $-7$-twists. Using Cremona's tables \cite{cremona}, these are

	\begin{tabular}{|r|r|r|}
		\hline
		Label & $j$-invariant & Discriminant \\ \hline
		$1369b1$	& $3^3 \cdot 37$ & $-37^8$ \\
		$1369b2$	& $-3^3 \cdot 37 \cdot 719^3$ & $-37^8$ \\
		$1369c1$ & $3^3 \cdot 37$ & $-37^2$ \\
		$1369c2$ & $-3^3 \cdot 37 \cdot 719^3$ & $-37^2$ \\

		$67081b1$ & $3^3 \cdot 37$ & $-7^6 \cdot 37^8$ \\
		$67081b2$ & $-3^3 \cdot 37 \cdot 719^3$ & $-7^6 \cdot 37^8$ \\
		$67081d1$ & $3^3 \cdot 37$ & $-7^6 \cdot 37^2$ \\
		$67081d2$ & $-3^3 \cdot 37 \cdot 719^3$ & $-7^6 \cdot 37^2$ \\

		$3969a1$ & $3^2 \cdot 7^4$ & $3^4 \cdot 7^8$ \\
		$3969a2$ & $3^2 \cdot 7 \cdot 2647^3$ & $3^4 \cdot 7^8$ \\
		$3969c1$ & $3^2 \cdot 7^4$ & $3^4 \cdot 7^2$ \\
		$3969c2$ & $3^2 \cdot 7 \cdot 2647^3$ & $3^4 \cdot 7^2$ \\
		$3969e1$ & $3^2 \cdot 7^4$ & $3^{10} \cdot 7^2$ \\
		$3969e2$ & $3^2 \cdot 7 \cdot 2647^3$ & $3^{10} \cdot 7^2$ \\
		$3969f1$ & $3^2 \cdot 7^4$ & $3^{10} \cdot 7^2$ \\
		$3969f2$ & $3^2 \cdot 7 \cdot 2647^3$ & $3^{10} \cdot 7^2$ \\
		\hline
	\end{tabular}

Next, we use Theorem \ref{mikorparos} to rule out rows with $37^2$ and $3^{10}$ in the discriminant.

\begin{thm}
\label{utolsotetel}

Assume (I)-(IV) for an elliptic curve $E/K$ having rational coefficients. Then
$$\tau:=\rk_{\Lambda(H)} X(E/\Kcc) \ge 2$$
holds with finitely many exceptions up to $\QQ$-isomorphism of elliptic curves.
The possibly exceptional isomorphism classes are classified by the following table. \\

	\begin{tabular}{|r|r|r|r|r|}
		\hline
	$p$ & $E$ (label) & $j$-invariant & Discriminant & rank over $\QQ$ \\ \hline
	$7$	&$1369b1$	& $3^3 \cdot 37$ & $-37^8$ & $1$ \\
	$7$	&$1369b2$	& $-3^3 \cdot 37 \cdot 719^3$ & $-37^8$ & $1$ \\
	$7$	&$67081b1$ & $3^3 \cdot 37$ & $-7^6 \cdot 37^8$ & $0$ \\
	$7$	&$67081b2$ & $-3^3 \cdot 37 \cdot 719^3$ & $-7^6 \cdot 37^8$ & $0$ \\ \hline

	$7$	&$3969a1$ & $3^2 \cdot 7^4$ & $3^4 \cdot 7^8$ & $1$ \\
	$7$	&$3969a2$ & $3^2 \cdot 7 \cdot 2647^3$ & $3^4 \cdot 7^8$ & $1$ \\
	$7$	&$3969c1$ & $3^2 \cdot 7^4$ & $3^4 \cdot 7^2$ & $0$ \\
	$7$	&$3969c2$ & $3^2 \cdot 7 \cdot 2647^3$ & $3^4 \cdot 7^2$ & $0$ \\
		\hline
	\end{tabular}
\end{thm}

\vspace{1em}

Note that the first and second four curves in this table form two equivalence classes: these are isomorphic or $7$-isogenous over $\QQ(\sqrt{-7}) \le \QQ(\mu_7) \le K$ (for any possible $K$) and since $\lambda$ and $s_{E/\Kc}$ are isogeny invariants, these have the same $\tau$ given assumptions (I)-(IV).

\begin{remark}
From these data it follows that all these curves have rank $1$ over $\QQ(\sqrt{-7})$ which is a necessary condition for $\tau=1$. Further Iwasawa theoretic calculations would be needed to compute their $\lambda$ rank (which equals their $\tau$ rank).
\end{remark}


\vspace{4em}

\end{document}